\documentclass[10pt,a4paper]{article}

\usepackage[hyperfootnotes=false,pdfpage layout=SinglePage]{hyperref}
\usepackage{mathptmx}
\DeclareMathAlphabet{\mathcal}{OMS}{cmsy}{m}{n}
\usepackage{amssymb}
\usepackage{amsmath}
\usepackage{amsthm}
\usepackage{latexsym}
\usepackage{amsfonts}
\usepackage{mathrsfs}
\usepackage{setspace}
\usepackage{stmaryrd}

\newtheorem{thm}{Theorem}[section]
\newtheorem{cor}[thm]{Corollary}
\newtheorem{lem}[thm]{Lemma}
\newtheorem{prop}[thm]{Proposition}
\theoremstyle{definition}

\newtheorem{rem}[thm]{Remark}
\newtheorem{ex}[thm]{Example}

\numberwithin{equation}{section}
\renewcommand{\emptyset}{\varnothing}

\newcommand{\R}{\ensuremath{\mathbb R}}    
\newcommand{\C}{\ensuremath{\mathbb C}}    
\newcommand{\N}{\ensuremath{\mathbb N}}    

\newcommand{\gperp}{{[\perp]}}

\newcommand{\product}{[\cdot\,,\cdot]}
\newcommand{\hproduct}{(\cdot\,,\cdot)}
\newcommand{\lk}{\langle}
\newcommand{\rk}{\rangle}

\newcommand{\calH}{\mathcal H}

\newcommand{\calK}{\mathcal K}         
\newcommand{\calL}{\mathcal L}         
\newcommand{\calM}{\mathcal M}

         \newcommand{\frakR}{\mathfrak R}
\newcommand{\calS}{\mathcal S}

\newcommand{\la}{\lambda}
\newcommand{\veps}{\varepsilon}

\newcommand{\bmat}{\begin{pmatrix}}
\newcommand{\emat}{\end{pmatrix}}
\newcommand{\mat}[4]
{
   \begin{pmatrix}
      #1 & #2\\
      #3 & #4
   \end{pmatrix}
}

\renewcommand{\Re}{\operatorname{Re}}
\newcommand{\linspan}{\operatorname{span}}
\renewcommand{\ker}{\operatorname{ker}}
\newcommand{\ran}{\operatorname{ran}}
\newcommand{\dom}{\operatorname{dom}}

\renewcommand{\sp}{\sigma_{+}}
\newcommand{\sm}{\sigma_{-}}

\newcommand{\Lra}{\Longrightarrow}

\newcommand{\upto}{\uparrow}
\newcommand{\downto}{\downarrow}

\newcommand{\ol}{\overline}
\newcommand{\ds}{\dotplus}
\newcommand{\wt}{\widetilde}

\newcommand{\ov}{\overline}

\newcommand{\III}{\interleave} 

\begin{document}
\vspace*{-.3cm}
\begin{center}
\begin{spacing}{1.7}
{\LARGE\bf The numerical range of non-negative operators in Krein spaces}
\end{spacing}

\vspace{.5cm}
\begin{spacing}{1.7}
{\large Friedrich Philipp and Carsten Trunk}
\end{spacing}
\end{center}

\vspace{.1cm}\hrule\vspace*{.4cm}
\noindent{\bf Abstract}

\vspace*{.3cm}\noindent
We define and characterize the Krein space numerical range $W(A)$ and the Krein space co-numeri\-cal range $W_{\rm co}(A)$ of a non-negative operator $A$ in a Krein space. It is shown that the non-zero spectrum of $A$ is contained in the closure of $W(A)\cap W_{\rm co}(A)$.
\vspace*{0.4cm}\hrule
%
%

\vspace*{.5cm}
\section{Introduction}
The classical numerical range of an operator in a Hilbert space has
been studied extensively and there are many results which connect
algebraic and analytic properties  of an operator with the geometric
properties of its numerical range. For an operator $A$ acting in a Krein
space $(\cal K, \product)$ the Krein space numerical range is defined
by
$$
W(A) = \left\{\frac{[Ax,x]}{[x,x]} : x\in\dom A,\,[x,x]\ne 0\right\}.
$$
There is a substantial interest in studying these relations in the
Krein space setting, see,
e.g., \cite{ba1,ba2,be1,be2,be3,be4,wc,GP,Li1,Li2,Li3}. It is well-known \cite{ba1,ba2,GP} that each of the sets
$$
W^\pm(A) = \left\{[Ax,x] : x\in\dom A,\,[x,x]=\pm1\right\}
$$
is convex and, as $W(A) = W^+(A)\cup W^-(A)$,  $W(A)$ decomposes
into at most two convex subsets.
Using the joint numerical range, conditions for
 $W^\pm(A)$ to be contained in a half space or in a line
are given in \cite{Li2}, and in \cite{Li3} it is shown that the Krein space numerical range is pseudo-convex\footnote{A set is called pseudo-convex (\cite{be2,Li3}) if for any pair of distinct points $x,y$ in this set either the closed line segment joining them or the straight line through $x$ and $y$ except the open line segment joining $x$ and $y$ is contained in the set.} for a special class of matrices. 
In  \cite{be1} boundary generating curves, corners and computer
generation of  the Krein space numerical range are investigated, in
\cite{be2,be3} relations between
the sets $W^+(A)$ and  $W^-(A)$ are discussed and in
\cite{be4} the numerical range is
completely characterized in  2-dimensional Krein spaces.
 Moreover, in \cite{Li1} operators with bounded Krein space numerical ranges
are studied.

In the present paper we give a complete
description of $W(A)$ for non-negative operators in Krein spaces. In fact, it turns out
that $W(A)$ always consists of the entire real axis with the possible exception of
a bounded interval $\Delta$ with $0\in\ol\Delta$. The boundary points of this interval can be calculated
in terms of the positive/negative spectrum of  $A$, see Theorem \ref{t:w} below. Hence,
$W(A)$ is a pseudo-convex set. This (partially) extends a result in 
\cite[Proposition 2.3]{be2} where it is shown that $W(A)$ is pseudo-convex if $A$ is a Krein
space normal matrix with simple eigenvalues only such that Re$A$ has only real eigenvalues.

In the recent paper \cite{wc} D.\ Wu and A.\ Chen proved with elementary methods that the
spectrum of a non-negative operator $A$ in a Krein space $(\calK,\product)$ is always
contained in the closure of its Krein space numerical range,
\begin{equation}\label{Rofe}
\sigma(A)\subset\ol{W(A)}.
\end{equation}
In fact, this statement follows almost immediately with the help of the spectral function
$E$ (see \cite{l}) of the operator $A$: If, e.g., $\la\in\sigma(A)\cap (0,\infty)$ and
$\Delta\subset (0,\infty)$ is a compact interval with $\la$ in its interior, then
$A|E(\Delta)\calK$ is selfadjoint in the Hilbert space $(E(\Delta)\calK,\product)$ and
hence $\la\in \ol{W(A|E(\Delta)\calK)}\subset\ol{W(A)}$. A similar argumentation applies
to negative points in the spectrum of $A$. Hence, it remains to consider the point zero
 in the case when it is an isolated point of $\sigma(A)$. Then either $0\in W(A)$ or
$\ker A$ is neutral. If $\ker A$ is neutral, then we find neutral vectors $x_1$ and $x_0$ such that
$[x_0,x_1] = 1$, $Ax_1 = x_0$ and $Ax_0 = 0$ (cf.\ \eqref{e:carsten} below). Setting $x = tx_1+ x_0$ we obtain
$[Ax,x]/[x,x] = t^2/2$ which tends to zero as $t$ tends to zero.\footnote{We would like
 to mention that in \cite{wc} there is a mistake in the proof of $0\in\ol{W(A)}$ if
$0\in\sigma(A)$, cf.\ Remark \ref{HeuteIstSportfest} below.}

The spectral
inclusion \eqref{Rofe}
 is not very useful since the numerical range
$W(A)$ is always neither bounded from above nor from below. If the inner product $\product$
is not definite on $\ker A$, then  $W(A)$ even covers the entire
real line (with the possible exception of zero). The following
simple example illustrates this.
\begin{ex}
In $\calK := \C^2$ denote by $\hproduct$ the standard scalar product and define the
following matrices
$$
J := \mat 0 1 1 0\quad\text{and}\quad A := \mat 0 1 0 0.
$$
Then $\product := (J\cdot,\cdot)$ defines a Krein space inner product on $\calK$ and
$A$ is non-negative in $(\calK,\product)$. Moreover, for $x = (x_1,x_2)^T\in\C^2$ we have
 $[x,x] =  x_2\ol{x_1} + x_1\ol{x_2} =
2\Re(x_2\ol{x_1})$ and $[Ax,x]=(JAx,x)=|x_2|^2$. Hence, $x_2 = 1$ and $x_1 = t\in\R\setminus\{0\}$
give $W(A) = \R\setminus\{0\}$.
\end{ex}

For this reason we define in Section \ref{s:cow} another subset of the real line which
is connected with $A$: The co-numerical range
$$
W_{\rm co}(A) := \left\{\frac{[Ax,Ax]}{[Ax,x]} : x\in\dom A,\,[Ax,x]\ne 0\right\},
$$
and we  show  in Section \ref{s:cow} the following spectral inclusion:
\begin{equation*}
\sigma(A)\,\subset\,\ol{W(A)\cap W_{\rm co}(A)},
\end{equation*}
with one exception in a very special case in which the inclusion only holds for the set $\sigma(A)\setminus\{0\}$, cf.\ Theorem \ref{t:spec_inc} below.

\section{Non-negative operators in Krein spaces}
Throughout this note let $(\calK,\product)$ be a Krein space, i.e.\
a vector space $\calK$ with a Hermitian
non-degenerate sesquilinear form   $\product$ which admits
 a so-called {\it fundamental decomposition}
\begin{equation}\label{e:fd}
\calK = \calK_+\,[\ds]\,\calK_-,
\end{equation}
where $(\calK_\pm,\pm\product)$ are Hilbert spaces. The symbol $[\ds]$ denotes the direct and
$\product$-orthogonal sum, i.e.\ $\calK_+=\calK_-^\gperp$ and
$\calK_-=\calK_+^\gperp$, where $\gperp$ denotes the $\product$-orthogonal
companion. Then
\begin{equation*}
(x,y) := [x_+,y_+] - [x_-,y_-],\quad x=x_++x_-,y=y_++y_-,\;x_\pm,y_\pm\in\calK_\pm,
\end{equation*}
is an inner product and $(\calK,\hproduct)$ is a Hilbert space.
Evidently, there exist infinitely many fundamental decompositions,
and each of them induces a Hilbert space norm as above. Any two such norms are equivalent, see
\cite[Proposition I.1.2]{l}. Therefore, all topological notions are understood with respect to
the topology induced by these norms. 
For a detailed study of Krein spaces and operators therein we refer to  \cite{ai,b,l}.

A vector $x\in\calK$, $x\neq 0$, is called {\it positive} ({\it negative})
if $[x,x]$ is positive (negative, respectively), and {\it neutral} if $[x,x]=0$. A subspace $\calL\subset\calK$
is called positive (negative, neutral) if each $x\in\calL\setminus\{0\}$ is positive
(negative, neutral, respectively). Furthermore, the subspace $\calL$ is called
{\it non-negative} ({\it non-positive}) if each $x\in\calL\setminus\{0\}$ is either
 neutral or positive (negative, respectively). In addition, we say that the subspace
 $\calL$ is {\it definite} if it is positive or negative. A subspace is called
{\it indefinite} if it is not definite. Note that the trivial subspace $\{0\}$ is
positive, negative and neutral and therefore not indefinite.

The ({\it Krein space{\rm )} adjoint} $T^+$ of a densely defined linear operator $T$ in $(\calK,\product)$ has
domain
$$
\dom T^+ := \{y\in\calK : \exists u\in\calK\forall x\in\dom T : [Tx,y] = [x,u]\}
$$
and satisfies
$$
[Tx,y] = [x,T^+y]\quad\text{for all }x\in\dom T,\,y\in\dom T^+.
$$
The operator $T$ is called {\it selfadjoint} if $T = T^+$.
Note that the spectrum of a selfadjoint operator in a Krein space is in general
not a subset of $\R$. However, the {\it non-negative} operators
in $(\calK,\product)$ have only real spectrum  and no residual
spectrum (cf.\ \cite[Corollary IV.6.2]{b}). Here,
a selfadjoint operator $A$ in $(\calK,\product)$ is called {\em non-negative} if
$\rho(A)\neq\emptyset$ and if $[Ax,x]\ge 0$ holds for all $x\in\dom A$.
A non-negative operator is a special type of a definitizable operator,
see \cite{l}.

 Recall that a non-negative operator
$A$ in $(\calK,\product)$ possesses a spectral function $E$ on $\R$ with the possible
singularities $0$ and $\infty$. The spectral projection $E(\Delta)$ is selfadjoint
in $(\calK,\product)$
and is defined for all bounded Borel sets $\Delta\subset\R$ with $0\notin\partial\Delta$
and their complements $\R\setminus\Delta$. We denote the collection of these sets by $\frakR$. 
The point zero is called a {\it critical point} of $A$ if for each
$\veps > 0$ the subspace $E([-\veps,\veps])\calK$ is indefinite. Analogously, the point
$\infty$ is called a critical point of $A$ if for each $C > 0$ the subspace
$E(\R\setminus (-C,C))\calK$ is indefinite.

Below, we will make extensive use of the spectral function. In the following we collect some of its properties (see e.g., \cite{l}). Let $\Delta,\Delta_1,\Delta_2,\ldots\in\frakR$. Then
\begin{enumerate}
\item[{\rm (a)}] $E(\Delta)$ is a bounded selfadjoint projection in $(\calK,\product)$ and commutes with every bounded operator which commutes with the resolvent of $A$;
\item[{\rm (b)}] If the $\Delta_j$, $j\in\N$, are mutually disjoint, and if their union is an element of $\frakR$, then
$$
E\left(\bigcup_{j=1}^\infty\Delta_j\right)x = \sum_{j=1}^\infty E(\Delta_j)x
$$
for every $x\in\calK$;
\item[{\rm (c)}] $E(\Delta_1 \cap \Delta_2)=E(\Delta_1)E(\Delta_2)$;
\item[{\rm (d)}] $\sigma(A|E(\Delta)\calK)\,\subset\,\sigma(A)\cap\ol\Delta$ \;\;\;and\;\; $\sigma(A|(I - E(\Delta))\calH)\,\subset\,\ol{\sigma(A)\setminus\Delta}$\,;
\item[{\rm (e)}] If $\Delta$ is bounded, then $E(\Delta)\calK$ is a subset of $\dom A$ and $A|E(\Delta)\cal K$
 is a bounded operator.
\end{enumerate}

Note that (d) implies that $E(\mathbb R)=I$ and $E(\emptyset)=0$.

A point $\la\in\sigma(A)$ is said to be a spectral point of {\it positive}
({\it negative}) {\it type} of $A$ if there exists an open neighborhood
$\Delta$ of $\la$ such that $(E(\Delta)\calK, \product)$
($(E(\Delta)\calK, -\product)$, respectively)
is a Hilbert space. The set consisting of all spectral points of
positive (negative) type of $A$ is denoted by $\sp(A)$ ($\sm(A)$, respectively).
We have\footnote{We use the notations
$\R^+ := (0,\infty)$
and $\R^- := (-\infty,0)$.}
\begin{equation}\label{HerzHerz}
\R^\pm\cap\sigma(A)\,\subset\,\sigma_\pm(A).
\end{equation}

We mention the following relation which holds for all $x\in\dom A$:
\begin{equation}\label{e:carsten}
[Ax,x] = 0\quad\Lra\quad x\in\ker A.
\end{equation}
Indeed, the application of the Cauchy-Bunyakowski inequality to the semi-definite inner
product $[A\cdot,\cdot]$ gives $\big|[Ax,y]\big|^2\,\le\,[Ax,x][Ay,y]$ for all $x,y\in\dom A$,
and \eqref{e:carsten} follows.

In the next lemma we collect some statements on the spectral properties of
the point zero. These are well-known, cf.\ \cite[Proposition II.2.1 and Section II.5]{l}.

\begin{lem}\label{l:super_basic}
The length of a Jordan chain corresponding to the eigenvalue zero of
 a non-negative operator $A$ is at most $2$
and the corresponding eigenvector of a Jordan chain of length two is neutral. 
If zero is an isolated spectral point of $A$ then it is an eigenvalue. 
Set\footnote{Note that in \cite{l} the closed linear span is used in the
definition of $\calS_0^+$ and $\calS_0^-$.}
\begin{align*}
\calS_0   &:= \bigcap\left\{E(\Delta)\calK : \Delta\in\frakR,\,0\in\Delta\right\},\\
\calS_0^+ &:= \linspan\left\{E(\Delta)\calK :
\Delta\in\frakR,\,\ol\Delta\subset\R^+\right\},\\
\calS_0^- &:= \linspan\left\{E(\Delta)\calK :
\Delta\in\frakR,\,\ol\Delta\subset\R^-\right\}.
\end{align*}
Then $\calS_0^+$ is positive and $\calS_0^-$ is negative. Moreover,
$\calS_0^+$, $\calS_0^-$, $\calS_0$ are $A$-invariant and are contained
in $\dom A$, and $\calS_0$ is the root subspace of $A$ corresponding to zero.
We have
\begin{equation}\label{e:langer}
\calS_0 = (\calS_0^+\,[\ds]\,\calS_0^-)^\gperp.
\end{equation}
\end{lem}

\section{The numerical range of a non-negative operator}
The (Krein space) numerical range of a non-negative operator $A$ in a Krein space  $(\calK, \product)$ is defined by
$$
W(A) := \left\{\frac{[Ax,x]}{[x,x]} : x\in\dom A,\,[x,x]\ne 0\right\}.
$$

In order to formulate our results we define the following constants:
\begin{align}
\begin{split}\label{Grins1}
\mu_- &:=
\begin{cases}
\sup(\sigma(A)\cap\R^-) &\text{if }\sigma(A)\cap\R^-\ne\emptyset,\\
-\infty               &\text{otherwise},
\end{cases}\\
\mu_+ &:=
\begin{cases}
\inf(\sigma(A)\cap\R^+) &\text{if }\sigma(A)\cap\R^+\ne\emptyset,\\
+\infty               &\text{otherwise}.
\end{cases}\\
\nu_- &:=
\begin{cases}
\inf(\sigma(A)\cap\R^-) &\text{if }\sigma(A)\cap\R^-\ne\emptyset,\\
0                     &\text{otherwise},
\end{cases}\\
\nu_+ &:=
\begin{cases}
\sup(\sigma(A)\cap\R^+) &\text{if }\sigma(A)\cap\R^+\ne\emptyset,\\
0                     &\text{otherwise}.
\end{cases}
\end{split}
\end{align}

\begin{lem}\label{l:good_point}
Let $A$ be a non-negative operator in a Krein space  $(\calK, \product)$. Then the following statements hold.
\begin{enumerate}
\item[{\rm (i)}]  If $\mu_+ > 0$ and $\ker A$ is negative, then $0\in\sm(A)\cup\rho(A)$,
and for each $t\in [0,\mu_+]$ the operator $A - t$ is non-negative.
\item[{\rm (ii)}] If $\mu_- < 0$ and $\ker A$ is positive, then $0\in\sp(A)\cup\rho(A)$,
and for each $t\in [\mu_-,0]$ the operator $A - t$ is non-negative.
\end{enumerate}
\end{lem}
\begin{proof}
We will only prove (i). The proof of (ii) is similar. Let $\mu\in (0,\mu_+)$ be arbitrary.
As $0\in\sm(A)\cup\rho(A)$ if and only if $0\in\sm(B)\cup\rho(B)$,
$B := A|E((-\mu,\mu))\calK$, for the first assertion we may assume that $A$ is bounded and
 $\sigma(A)\cap\R^+ = \emptyset$. Therefore, $\calS_0^+ = \{0\}$. It follows from Lemma \ref{l:super_basic} that $\calS_0 = \ker A$. Due to
  \cite[Proposition I.1.1]{l}  $\calS_0^-$ is contained in a
  maximal non-positive
subspace $L_-$. By \cite[Theorem V.4.4]{b}
$L_-^\gperp$ is maximal non-negative and
$L_-^\gperp\subset \calS_0^{-\gperp} = \ker A$, see \eqref{e:langer}. 
Since $\ker A$ is negative, this implies $L_-^\gperp = \{0\}$ and  therefore $\calK_+ = \{0\}$ in \eqref{e:fd}.
Hence, $0\in\sm(A)\cup\rho(A)$.

In view of \eqref{e:langer} and \cite[Theorem I.5.2]{l}, we obtain the following decomposition
\begin{equation}\label{e:lem_dec}
\calK = \ker A \,[\ds]\,\ol{\calS_0^-\,[\ds]\,\calS_0^+}.
\end{equation}
It remains to prove that $A - t$ is non-negative for all $t\in [0,\mu_+]$. For this,
it suffices to consider only $t\in (0,\mu_+)$. For $x\in\ker A$ or $x\in\calS_0^-$ we have $[x,x]\leq 0$ and thus
$$
[(A - t)x,x] = [Ax,x] - t[x,x]\ge 0.
$$
If $x\in\calS_0^+$, then there exists a compact interval $\Delta\subset [\mu_+,\infty)$
such that $x\in E(\Delta)\calK$. Therefore,
$A|E(\Delta)\calK$ is a selfadjoint operator in the Hilbert space
$(E(\Delta)\calK,\product)$ and
$$
[(A - t)x,x] = [(A|E(\Delta)\calK)x,x] - t[x,x]\ge (\mu_+ - t)[x,x]\ge 0.
$$
Summing up, $[(A - t)x,x]\ge 0$ is valid for all $x\in\ker A \,[\ds]\,\calS_0^-\,[\ds]\,\calS_0^+$.

Moreover, for each compact interval $\Delta\subset \mathbb R \setminus (0,\mu_+)$ we
have $E(\Delta)\calK \subset \dom A$ and $(A-t)|E(\Delta)
\calK : E(\Delta)
\calK \to E(\Delta)
\calK$ is a bijective and boundedly invertible operator. Thus,
$$
 \ker A \,[\ds]\,\calS_0^-\,[\ds]\,\calS_0^+ \subset \left\{
 (A-t)x | x\in
  \ker A \,[\ds]\,\calS_0^-\,[\ds]\,\calS_0^+  \right\}
$$
and both sets are dense in $\calK$, see \eqref{e:lem_dec}. If for $x\in\ker A \,[\ds]\,\calS_0^-\,[\ds]\,\calS_0^+$ we set $y:= (A-t)x$, we obtain $[(A - t)^{-1}y,y] = [(A - t)x,x]\ge 0$ for all $y$ in a dense subset of $\calK$. But this is equivalent to the fact that $A-t$ is non-negative.
\end{proof}

\begin{rem}\label{r:infty}
It follows in particular from Lemma \ref{l:good_point} that the point $0$ is not a
critical point of $A$ if $0\notin\sigma_p(A)$ and if either $\mu_+ > 0$ or $\mu_- < 0$.
Analogously, if $\nu_+ < \infty$ or $\nu_- > -\infty$, then the point $\infty$ is not a
critical point of $A$.
\end{rem}

The following theorem is the main result in this section.
It characterizes the numerical range of $A$.

\begin{thm}\label{t:w}
Let $A\neq 0$ be a non-negative operator in the Krein space $(\calK,\product)$ and let $(\calK, \product)$ be indefinite.
Then the following statements hold.
\begin{enumerate}
\item[{\rm (i)}]   If $\ker A = \{0\}$, then $\mu_->-\infty$,
$\mu_+<\infty$ and
$$
W(A)\cup\{\mu_-,\mu_+\} = (-\infty,\mu_-]\cup [\mu_+,\infty).
$$
\item[{\rm (ii)}]  If $\ker A$ is indefinite, then
$$
W(A)\cup\{0\} = \R.
$$
\item[{\rm (iii)}] If $\ker A\neq\{0\}$ is positive, then $\mu_- > -\infty$ and
$$
W(A)\cup\{\mu_-\} = (-\infty,\mu_-]\cup [0,\infty).
$$
\item[{\rm (iv)}]  If $\ker A\neq\{0\}$ is negative, then $\mu_+ < \infty$ and
$$
W(A)\cup\{\mu_+\} = (-\infty,0]\cup [\mu_+,\infty).
$$
\end{enumerate}
Moreover, the following holds for the points $\mu_-$, $\mu_+$ and $0$:
\begin{enumerate}
\item[{\rm (a)}] $0\in W(A)$ if and only if $0\in\sigma_p(A)$ and $\ker A$ is not neutral.
\item[{\rm (b)}] $\mu_-\in W(A)$ in cases {\rm (i)} and {\rm (iii)} if and only if
$\mu_-\in\sigma_p(A)$.
\item[{\rm (c)}] $\mu_+\in W(A)$ in cases {\rm (i)} and {\rm (iv)} if and only if
$\mu_+\in\sigma_p(A)$.
\end{enumerate}
\end{thm}
\begin{proof}
We begin with the proof of (ii). If $\ker A$ is indefinite, then there exists a neutral
element $x_0\in\ker A$, $x_0\neq 0$. Moreover, as $A\neq 0$, we have $\ker A\neq\calK$,
and hence the interior of $\ker A$ is empty so that $\calK\setminus\ker A$ is a dense set
in $\calK$. Therefore, there exists $y\in\calK\setminus\ker A$ such that $[x_0,y]\neq 0$.
We may assume $y\in\dom A$ and $[x_0,y] = 1$. Set $u_0 := y - ([y,y]/2)x_0\in\dom A$.
Then $[x_0,u_0] = 1$, $[u_0,u_0] = 0$ and $u_0\notin\ker A$, hence $[Au_0,u_0]\ne 0$,
see \eqref{e:carsten}. Moreover, for all $t\in\R\setminus\{0\}$ we have
$$
\frac{[A(tx_0+u_0),tx_0+u_0]}{[tx_0+u_0,tx_0+u_0]} = \frac{[Au_0,u_0]}{2t},
$$
which shows that $\R\setminus\{0\}\subset W(A)$ or, equivalently, $W(A)\cup\{0\} = \R$.

In what follows we assume that $\ker A$ is definite. Then Lemma
\ref{l:super_basic} implies $\calS_0 = \ker A$. Note
that in (i), (iii) and (iv) $\mu_+ = \infty$ and $\mu_- = -\infty$ is not possible, since this would imply
 $A = 0$. Without loss of generality we assume $\mu_+ < \infty$. If $\mu_- = -\infty$,
then $\ker A$ must be negative and $\neq\{0\}$ since otherwise $0\in\sp(A)$ by Lemma
\ref{l:good_point}(ii) and hence $\calK_- = \{0\}$ in \eqref{e:fd} which we had excluded.
Therefore, there exist $\la_-\in\sigma(A)\cap(-\infty,0]$ and $\la_+\in\sigma(A)\cap\R^+$.
 Let $\Delta_+$ be an open interval with $\la_+\in\Delta_+$ and $\ol{\Delta_+}\subset\R^+$
 and set $\calH_+ := E(\Delta_+)\calK$. If $\la_- < 0$ choose an open interval $\Delta_-$
with $\la_-\in\Delta_-$ and $\ol{\Delta_-}\subset\R^-$ and set
$\calH_- := E(\Delta_-)\calK$.
If $\la_- = 0$ choose a negative vector $x_-\in\ker A\setminus\{0\}$ and set
$\calH_- := \linspan\{x_-\}$. As $(\calH_\pm,\pm\product)$ are mutually orthogonal Hilbert
 spaces, we may choose a fundamental decomposition
\begin{equation}\label{e:fd2}
\calK = \calK_+'[\ds]\calK_-'
\end{equation}
such that $\calH_\pm\subset\calK_\pm'$, cf.\ \cite[Theorem V.3.5]{b}.
 By $\|\cdot\|$ denote the Hilbert space norm
 arising from \eqref{e:fd2}. Now let $(x_n^\pm)$ be sequences in $\calH_\pm$ with
 $\|x_n^\pm\| =\pm [x_n^\pm, x_n^\pm]= 1$, $x_n^\pm\in \dom A$
 for each $n\in\N$ and $(A - \la_\pm)x_n^\pm\to 0$ ($n\to\infty$)
 and define
$$
x_n := t x_n^+ + x_n^-
$$
with some $t\in\R^+$, $t\ne 1$. Then
$$
[x_n,x_n] = t^2[x_n^+,x_n^+] + [x_n^-,x_n^-] = t^2 - 1
$$
and
$$
[Ax_n,x_n] = t^2[Ax_n^+,x_n^+] + [Ax_n^-,x_n^-] = t^2\la_+ + |\la_-| + t^2\veps_n^+
 + \veps_n^-,
$$
where
$$
\veps_n^\pm = [(A - \la_\pm)x_n^\pm,x_n^\pm],\quad n\in\N.
$$
Since $\veps_n^\pm\to 0$ ($n\to\infty$), there exists $N\in\N$ such that $\la_+ +
 |\la_-| + \veps_n^+ + \veps_n^- > 0$ for $n\ge N$.
Therefore, it follows that for $n\ge N$ we have
$$
\frac{[Ax_n,x_n]}{[x_n,x_n]} = \frac{t^2\la_+ + |\la_-| + t^2\veps_n^+ +
 \veps_n^-}{t^2 - 1}\to
\begin{cases}
\la_- - \veps_n^- &\text{as }t\downto 0,\\
-\infty           &\text{as }t\upto 1,\\
+\infty           &\text{as }t\downto 1,\\
\la_+ + \veps_n^+ &\text{as }t\upto\infty.
\end{cases}
$$
This proves $(-\infty,\la_-)\cup(\la_+,\infty)\subset W(A)$ and thus
\begin{equation}\label{e:inclusion}
W(A)\,\supset\,
\begin{cases}
(-\infty,\mu_-)\cup(\mu_+,\infty) &\text{ if }\mu_-\neq -\infty,\,\mu_+\neq\infty\\
(-\infty,0)\cup(\mu_+,\infty)     &\text{ if }\mu_- = -\infty,\,\mu_+\neq\infty\\
(-\infty,\mu_-)\cup(0,\infty)     &\text{ if }\mu_-\neq -\infty,\,\mu_+ = \infty\,.
\end{cases}
\end{equation}
%
We will now prove (i).
By Lemma \ref{l:good_point} and the assumption that $\calK$ is indefinite we have
$\mu_->-\infty$ and $\mu_+<\infty$.
Due to \eqref{e:inclusion}, (i) clearly holds if $\mu_+ = \mu_- = 0$. Without loss of generality
we assume $\mu_- < 0$. By assumption, $0\notin\sigma_p(A)$, which implies that zero cannot be an
isolated spectral point of $A$, cf. Lemma \ref{l:super_basic}. Hence, $\mu_+ = 0$ or $0\in\rho(A)$. In the case $\mu_+ = 0$
we have $\ker A=\{0\}$ which is, by definition, simultaneously a
positive, negative and neutral subspace, and
for small $\veps > 0$
  the operator $A + \veps$ is non-negative (Lemma \ref{l:good_point}(ii)) with
$0\in\rho(A+\veps)$.
Since $W(A+\veps) = \{t+\veps : t\in W(A)\}$, it is no restriction to assume
$0\in\rho(A)$ and thus $\mu_+ > 0$.

Let $x_\pm\in\calS_0^\pm$, set $x := x_+ + x_-$ and assume $[x,x]\ne 0$. Then
there exists $t > 0$ such that $x_+\in \calM_+ := E([\mu_+,t])\calK$ and
$x_-\in \calM_- := E([-t,\mu_-])\calK$. Choose a fundamental decomposition
$\calK = \calK_+[\ds]\calK_-$ such that $\calM_\pm\subset\calK_\pm$
(cf.\ \cite[Theorem V.3.5]{b}) and denote the
 corresponding Hilbert space scalar product and norm by $\hproduct$ and $\|\cdot\|$, respectively.
If $\|x_+\| > \|x_-\|$, then
$$
\frac{[Ax,x]}{[x,x]} = \frac{[Ax_+,x_+] + [Ax_-,x_-]}{\|x_+\|^2 -
 \|x_-\|^2}\,\ge\,\frac{(Ax_+,x_+)}{\|x_+\|^2}\,\ge\,\mu_+.
$$
And if $\|x_+\| < \|x_-\|$, then
$$
\frac{[Ax,x]}{[x,x]} = -\frac{[Ax_+,x_+] + [Ax_-,x_-]}{\|x_-\|^2 -
 \|x_+\|^2}\,\le\,\frac{(Ax_-,x_-)}{\|x_-\|^2}\,\le\,\mu_-.
$$
This implies that $\frac{[Ax,x]}{[x,x]}\in (-\infty,\mu_-]\cup[\mu_+,\infty)$ for all
 $x\in\calS_0^+[\ds]\calS_0^-$ with $[x,x]\ne 0$.

Now, let $x\in\dom A$ such that $[x,x]\neq 0$ and set $y := Ax$. By \eqref{e:langer}
 the subspace $\calS_0^+[\ds]\calS_0^-$ is dense in $\calK$. Hence, there exists a
 sequence $(y_n)$ in $\calS_0^+[\ds]\calS_0^-$ such that $y_n\to y$ ($n\to\infty$).
 Since $A^{-1}\calS_0^\pm\subset\calS_0^\pm$ also the vectors $x_n := A^{-1}y_n$ are
elements of $\calS_0^+[\ds]\calS_0^-$, and $[x_n,x_n]\neq 0$ holds for all $n\ge N$
with some $N\in\N$. Therefore we obtain
$$
\frac{[Ax,x]}{[x,x]} = \frac{[y,A^{-1}y]}{[A^{-1}y,A^{-1}y]} =
\lim_{n\to\infty}\frac{[y_n,A^{-1}y_n]}{[A^{-1}y_n,A^{-1}y_n]} =
 \lim_{n\to\infty}\frac{[Ax_n,x_n]}{[x_n,x_n]}
$$
and thus $\frac{[Ax,x]}{[x,x]}\in (-\infty,\mu_-]\cup[\mu_+,\infty)$. Statement (i) is proved.

Now, assume that $\ker A\neq\{0\}$ is positive. Again, if $\mu_+ = \mu_- = 0$, the assertion follows from \eqref{e:inclusion}. Recall that $\mu_- = -\infty$ and Lemma \ref{l:good_point}(ii) imply $0\in\sigma_+(A)$ and $\calK_- = \{0\}$ in \eqref{e:fd} which contradicts the assumption that $\calK$ is indefinite. Hence $\mu_-\in (-\infty,0]$.

Assume $\mu_- < 0$. Then, by Lemma \ref{l:good_point}(ii), we have $0\in\sp(A)$ and $A + \veps$, $\veps = -\mu_-/2$, is non-negative. Hence, by (i) we have $W(A + \veps)\cup\{-\veps,\veps\} = (-\infty,-\veps]\cup [\veps,\infty)$. Moreover, $\veps$ is an eigenvalue of $A + \veps$ and so $W(A + \veps)\cup\{-\veps\} = (-\infty,-\veps]\cup [\veps,\infty)$ which implies
\begin{align*}
W(A)\cup\{\mu_-\}
&= \{t - \veps : t\in W(A + \veps)\}\cup\{\mu_-\}\\
&= \{t - \veps : t\in W(A + \veps)\cup\{-\veps\}\}\\
&= \{t - \veps : t\in (-\infty,-\veps]\cup [\veps,\infty)\}\\
&= (-\infty,\mu_-]\cup [0,\infty).
\end{align*}
It remains to consider the case $\mu_- = 0$ and $\mu_+ > 0$. According to
 \eqref{e:inclusion} nothing is to show if $\mu_+ = \infty$. Thus, assume
 $\mu_+\in (0,\infty)$. We have to show that $(0,\mu_+]\subset W(A)$. To
this end choose some compact interval $\Delta\subset\R^+$ such that
$E(\Delta)\calK\neq\{0\}$ and choose $x_+\in E(\Delta)\calK$, $x_+\neq 0$,
with $[x_+,x_+] = 1$. Then $[Ax_+,x_+]\ge\mu_+ > 0$. Moreover, choose $u_+\in\ker A$
with $[u_+,u_+] = 1$. Then with $x := tu_+ + x_+$, $t\ge 0$, we have
$$
\frac{[Ax,x]}{[x,x]} = \frac{[Ax_+,x_+]}{t^2 + 1},
$$
which shows that $(0,[Ax_+,x_+]]\subset W(A)$ and thus also $(0,\mu_+]\subset W(A)$ and (iii) is proved. Statement
(iv) follows with a similar reasoning.

In order to see that (a) holds we observe that $0\in W(A)$ if and only if there exists
$x\in\dom A$ such that $[Ax,x] = 0$ and $[x,x]\neq 0$. By \eqref{e:carsten} this holds
if and only if there exists $x\in\ker A$ with $[x,x]\neq 0$ and
 (a) is shown. If $\mu_+ = 0$, then (c) follows from (a).
Let $0 < \mu_+ < \infty$ in cases (i) or (iv). Then by Lemma \ref{l:good_point}(i) the
operator $A - \mu_+$ is non-negative and $\mu_+\in\sigma_+(A)$. Hence, we have $\mu_+\in W(A)$ 
if and only if $0\in W(A - \mu_+)$, which, by (a), is equivalent to $\ker(A - \mu_+)\neq\{0\}$.
The statement (b) is proved similarly.
\end{proof}

\begin{rem}\label{HeuteIstSportfest}
 In the proof of Theorem 3.1 in \cite{wc} there is a mistake.
Contrary to the claim in  \cite{wc},
the sequence $[AJx_n,Jx_n]$ in the proof of Theorem 3.1
in \cite{wc}, does, in general,
not converge to zero. However, the statement of Theorem 3.1 in \cite{wc}
is correct, see also Corollary \ref{c:easy} below.
\end{rem}

\begin{cor}\label{c:easy}
If $\calK$ is indefinite and $A\neq 0$, then
\begin{enumerate}
\item[{\rm (i)}]   $\sigma(A)\subset\ol{W(A)}$.
\item[{\rm (ii)}]  The sets $W(A)\cap\R^+$ and $W(A)\cap\R^-$ are convex and unbounded.
\item[{\rm (iii)}] If $\mu_+ = \mu_- = 0$ then $\ol{W(A)} = \R$.
\end{enumerate}
\end{cor}

\section{The co-numerical range}\label{s:cow}
It follows from Theorem \ref{t:w} that the Krein space numerical range $W(A)$ of a non-negative operator $A$ is always neither bounded from above nor from below. The spectral inclusion $\sigma(A)\subset\ol{W(A)}$ in Corollary \ref{c:easy} is thus not very useful, especially when the operator $A$ is bounded. For this reason we next define the {\it co-numerical range} of the non-negative operator $A$ in the Krein space $(\calK,\product)$ by
$$
W_{\rm co}(A) := \left\{\frac{[Ax,Ax]}{[Ax,x]} : x\in\dom A,\,Ax\ne 0\right\}.
$$
To motivate this definition, assume that the operator $A$ is bounded and boundedly invertible. Then $(\calK,[A\cdot,\cdot])$ is a Hilbert space, and $W_{\rm co}(A)$ is just the numerical range of the selfadjoint operator $A$ in this Hilbert space. Thus, $W_{\rm co}(A)\setminus\{\nu_-,\nu_+\} = (\nu_-,\nu_+)$ which locates the spectrum of $A$ much better than the numerical range $W(A)$ which, in this case, satisfies $W(A)\setminus\{\mu_-,\mu_+\} = \R\setminus [\mu_-,\mu_+]$. The main result in this section, Theorem \ref{t:cow}, generalizes the above observation.


The next lemma strengthens the statement in \eqref{e:carsten} in the case when $A$
 is bounded.

\begin{lem}\label{l:stronger}
Assume that $A$ is a bounded non-negative operator in the Krein space $(\calK,\product)$
  and let $(x_n)$   be a bounded sequence in  $\calK$ such that $[Ax_n,x_n]\to 0$ as
 $n\to\infty$. Then $Ax_n\to 0$ as $n\to\infty$.
\end{lem}
\begin{proof}
Let $J$ be the fundamental symmetry corresponding to the fundamental decomposition
 \eqref{e:fd}. An application of the Cauchy-Bunyakowski inequality to the semi-definite
 inner product $[A\cdot,\cdot]$ gives
$$
\|Ax_n\|^2 = [Ax_n,JAx_n] \leq [Ax_n,x_n] [AJAx_n,JAx_n],
$$
which tends to zero as $n\to\infty$.
\end{proof}

Similar techniques and ideas as in the following proposition and its proof can be found in, e.g., \cite{l} and \cite{a}.

\begin{prop}\label{p:ext}
Assume that $A$ is a bounded non-negative operator in the Krein space $(\calK,\product)$. On the space $\calK_0 := \calK/\ker A$ define the inner product $\lk\cdot\,,\cdot\rk$ and the operator $A_0$ by
$$
\lk[x],[y]\rk := [Ax,y]\quad\text{and}\quad A_0[x] := [Ax],\quad x,y\in\calK,
$$
respectively. Then $(\calK_0,\lk\cdot\,,\cdot\rk)$ is a pre-Hilbert space, and the operator $A_0$ is bounded and symmetric in $(\calK_0,\lk\cdot\,,\cdot\rk)$. By $\wt\calK$ denote the completion of $(\calK_0,\lk\cdot\,,\cdot\rk)$ and by $\wt A$ the bounded selfadjoint extension of $A_0$ in $\wt\calK$. Then we have
\begin{equation}\label{e:sigmas}
\sigma(\wt A)\setminus\{0\} = \sigma(A)\setminus\{0\}.
\end{equation}
Moreover, $0\in\rho(\wt A)$ if and only if either $0\in\rho(A)$ or zero is an isolated eigenvalue of $A$ such that $\ker A = \ker A^2$.
\end{prop}
\begin{proof}
It is evident that $(\calK_0,\lk\cdot\,,\cdot\rk)$ is a
pre-Hilbert space and that $A_0$ is symmetric with respect
to $\lk\cdot\,,\cdot\rk$. By $\III\cdot\III$ denote the norm on
$\calK_0$ induced by $\lk\cdot\,,\cdot\rk$. For the boundedness of $A_0$
let $x\in\calK$ and let $J$ be the fundamental symmetry corresponding to the
fundamental decomposition \eqref{e:fd}.
Then apply Reid's inequality (see, e.g., \cite{r}) to
the operators  $S := JA$ and $K := A^2$ to obtain
$$
\III A_0[x]\III^2 = \III[Ax]\III^2 = [A^3x,x] = (SKx,x)\le\|A^2\|[Ax,x]
\le\|A\|^2\III[x]\III^2,
$$
and $\wt A$ is a bounded operator.

In order to prove the inclusion $\sigma(A)\setminus\{0\}\subset\sigma(\wt A)$
let $\la\in\sigma(A)\setminus\{0\}$. Then there exists a sequence $(x_n)$ in
 $\calK$ with
$\|x_n\| = 1$ for all $n\in\N$ and $(A - \la)x_n\to 0$ as $n\to\infty$. Hence,
$$
\III(A_0 - \la)[x_n]\III^2 = [A(A - \la)x_n,(A - \la)x_n]\to 0
$$
as $n\to\infty$. Assume $\liminf_{n\to\infty}\III[x_n]\III = 0$.
Then for a subsequence $(x_{n_k})$ of $(x_n)$
we have $[Ax_{n_k},x_{n_k}]\to 0$ and, by Lemma \ref{l:stronger},
$Ax_{n_k}\to 0$ as $k\to\infty$.
As this is not possible due
to $(A - \la)x_n\to 0$ as $n\to\infty$ and $\la\neq 0$, we obtain $\la\in\sigma(\wt A)$.

Contrary, let $\la\in\sigma(\wt A)\setminus\{0\}$.
Then there exists a sequence $([x_n])$ in $\calK_0$ with
 $\III[x_n]\III = 1$ for each $n\in\N$ and $\III (A_0 - \la)[x_n]\III\to 0$ as
$n\to\infty$. That is,
$$
[Ax_n,x_n] = 1 \text{ and}
\quad [A(A - \la)x_n,(A - \la)x_n]\to 0\,\text{ as }n\to\infty.
$$
The second relation, together with Lemma \ref{l:stronger}, implies
\begin{equation*}
(A - \la)Ax_n\to 0\quad\text{as }n\to\infty.
\end{equation*}
Set $y_n:=Ax_n$.  Assume $\liminf_{n\to\infty}\|y_n\| = 0$.
Then for a subsequence $(y_{n_k})$ of $(y_n)$
we have $\|y_{n_k}\|\to 0$ and
$\III [y_{n_k}]\III^2 \leq \|A\| \|y_{n_k}\|^2\to 0$ as  $n\to\infty$. Hence
$$
0= \liminf_{k\to\infty}\III (A_0 - \la)[x_{n_k}]\III =
 \liminf_{k\to\infty}\III [y_{n_k}]  - \la[x_{n_k}]\III = |\lambda |,
$$
a contradiction. Relation \eqref{e:sigmas} is shown.

Assume that $0\in\rho(A)$ or that zero is an isolated eigenvalue of $A$ such that $\ker A = \ker A^2$. Then the spectral subspace of $A$ corresponding to zero coincides with $\ker A$. Hence, the inner product space $(\ker A,\product)$ is a Krein space, and the spectral subspace of $A$ corresponding to the spectral set $\sigma(A)\setminus\{0\}$ coincides with $\ran A$. In particular, $\ran A$ is closed, and we have
$$
\calK = \ker A\,[\ds]\,\ran A.
$$
This implies that $\ran A^2 = A\ran A = \ran A$ is closed. Thus, there exists a $\delta > 0$ such that $\|A^2x\|\ge\delta\|x\|$ for all $x\in\ran A$. Suppose now that $0\in\sigma(\wt A)$. Then there exists a sequence $([x_n])$ in $\calK_0$ with $\III[x_n]\III = 1$ for $n\in\N$ and $\III A_0[x_n]\III\to 0$ as $n\to\infty$. It is no restriction to assume $x_n\in\ran A$, $n\in\N$. Then $[A^2x_n,Ax_n]\to 0$ as $n\to\infty$, and Lemma \ref{l:stronger} implies $A^2x_n\to 0$ as $n\to\infty$. Thus, we obtain $x_n\to 0$ as $n\to\infty$ and
$$
1 = \III[x_n]\III^2 = [Ax_n,x_n]\to 0\quad\text{as }n\to\infty,
$$
which is a contradiction. Therefore, $0\in\rho(\wt A)$.

Conversely, assume that $0\in\rho(\wt A)\cap\sigma(A)$. Then by \eqref{e:sigmas} and Lemma \ref{l:super_basic} zero is an isolated eigenvalue of $A$. Suppose that there exists $x_0$ in the kernel of $A$ with $Ax_1 = x_0$ for some $x_1\in\calK$. Then we have
\begin{equation*}
A_0[x_1] = [Ax_1] = [x_0] = [0].
\end{equation*}
By assumption, $\wt A$ is injective, hence $x_1\in\ker A$. Therefore, $\ker A = \ker A^2$.
\end{proof}

In the next corollary we characterize the closure of the co-numerical range of a bounded non-negative operator $A$ in terms of the spectra of $\wt A$ and $A$.

\begin{cor}\label{cor}
Let $A$ and $\wt A$ be as in Proposition {\rm\ref{p:ext}}. Then we have
\begin{equation}\label{e:Ilmenau}
\ov{W_{\rm co}(A)}= [\min\sigma(\wt A),\max\sigma(\wt A)].
\end{equation}
If zero is not an isolated eigenvalue of $A$ or if both $\sigma(A)\cap\R^+$ and $\sigma(A)\cap\R^-$ are non-empty, then the following two relations hold:
\begin{enumerate}
\item[{\rm (a)}] $\min\sigma(\wt A) = \min\sigma(A)$.
\item[{\rm (b)}] $\max\sigma(\wt A) = \max\sigma(A)$.
\end{enumerate}
Let zero be an isolated eigenvalue of $A$. Then:
\begin{enumerate}
\item[{\rm (i)}]   If $\sigma(A)\cap\R^+ = \emptyset$ and $\sigma(A)\cap\R^-\neq\emptyset$, then {\rm (a)} holds, and {\rm (b)} holds if and only if $\ker A\neq\ker A^2$. Otherwise, $\max\sigma(\wt A) = \mu_- < 0$.
\item[{\rm (ii)}]  If $\sigma(A)\cap\R^+\neq\emptyset$ and $\sigma(A)\cap\R^- = \emptyset$, then {\rm (b)} holds, and {\rm (a)} holds if and only if $\ker A\neq\ker A^2$. Otherwise, $\min\sigma(\wt A) = \mu_+ > 0$.
\item[{\rm (iii)}] If $\sigma(A) = \{0\}$, then either $A = 0$ and hence $W_{\rm co}(A) = \emptyset$ or $\ker A\neq\ker A^2$ in which case $W_{\rm co}(A) = \{0\}$.
\end{enumerate}
\end{cor}
\begin{proof}
We have
\begin{align*}
\ov{W_{\rm co}(A)} & =\ov{ \left\{\frac{[Ax,Ax]}{[Ax,x]} : x\in \calK,\,Ax\ne 0\right\}}
= \ov{ \left\{\frac{\lk A_0[x],[x]\rk }{\lk[x],[x]\rk } :
x\in \calK,\,Ax\ne 0\right\}}\\[1ex]
&=\ov{ \left\{\frac{\lk A_0[x],[x]\rk }{\lk[x],[x]\rk } :
[x]\in \calK_0,\,[x]\ne [0]\right\}}
= [\min\sigma(\wt A), \max\sigma(\wt A)],
\end{align*}
and  \eqref{e:Ilmenau} is shown.
The last equality is a consequence of $\ol{A_0} = \wt A$ in $\wt\calK$ and well-known properties of the numerical range of a selfadjoint operator in a Hilbert space.

If both $\sigma(A)\cap\R^+$ and $\sigma(A)\cap\R^-$ are non-empty, then (a) and (b) follow directly from \eqref{e:sigmas}. Assume that $0\in\rho(A)$. Then, by Proposition \ref{p:ext}, we also have $0\in\rho(\wt A)$ and thus $\sigma(A) = \sigma(\wt A)$. In particular, both (a) and (b) are satisfied. Also, if $0\in\sigma(A)$ is not an isolated eigenvalue of $A$, Proposition \ref{p:ext} yields that $0\in\sigma(\wt A)$ and hence the validity of (a) and (b).

Now, assume that zero is an isolated eigenvalue of $A$. In order to prove (i), let $\sigma(A)\cap\R^+ = \emptyset$ and $\sigma(A)\cap\R^-\neq\emptyset$. Then, clearly, (a) holds, and (b) holds if and only if $0\in\sigma(\wt A)$. And as zero is an isolated eigenvalue of $A$, by Proposition \ref{p:ext} this is equivalent to $\ker A\neq\ker A^2$. If this is not the case, then $0\in\rho(\wt A)$, and $\max\sigma(\wt A) = \max(\sigma(A)\setminus\{0\}) = \mu_-$.

Finally, (ii) follows from (i), applied to $-A$ and $-\product$ instead of $A$ and $\product$, and (iii) is trivial.
\end{proof}

We now prove our main result on the co-numerical range of $A$.

\begin{thm}\label{t:cow}
Assume that $A$ is a non-negative operator in the Krein space $(\calK,\product)$.
Let $\calK$ be indefinite and $A\neq 0$. Then the following statements hold.
\begin{enumerate}
\item[{\rm (i)}]   If $\ran A$ is negative, then $\sigma(A)\cap\R^+ = \emptyset$ and
\begin{equation}\label{17}
W_{\rm co}(A)\setminus\{\nu_-,\mu_-\} = (\nu_-,\mu_-).
\end{equation}
\item[{\rm (ii)}]  If $\ran A$ is positive, then $\sigma(A)\cap\R^- = \emptyset$ and
$$
W_{\rm co}(A)\setminus\{\mu_+,\nu_+\} = (\mu_+,\nu_+).
$$
\item[{\rm (iii)}] If $\ran A$ is indefinite, then
\begin{equation*}
W_{\rm co}(A)\setminus\{\nu_-,\nu_+\} = (\nu_-,\nu_+).
\end{equation*}
\end{enumerate}
Moreover, the following holds for the points $\mu_\pm$, $\nu_\pm$ and $0$:
\begin{enumerate}
\item[{\rm (a)}] $0\in W_{\rm co}(A)$ if and only if $\ran A$ is indefinite.
\item[{\rm (b)}] $\mu_-\in W_{\rm co}(A)$ $(\nu_-\in W_{\rm co}(A))$ in {\rm (i)} if and only if $\mu_-\in\sigma_p(A)\setminus\{0\}$ $($resp. $\nu_-\in\sigma_p(A))$.
\item[{\rm (c)}] $\mu_+\in W_{\rm co}(A)$ $(\nu_+\in W_{\rm co}(A))$ in {\rm (ii)} if and only if
$\mu_+\in\sigma_p(A)\setminus\{0\}$ $($resp. $\nu_+\in\sigma_p(A))$.
\item[{\rm (d)}] $\nu_-\in W_{\rm co}(A)$ $(\nu_+\in W_{\rm co}(A))$ in {\rm (iii)} if and only if
$\nu_-\in\sigma_p(A)$ $($resp. $\nu_+\in\sigma_p(A))$.
\end{enumerate}
\end{thm}
\begin{proof}
First of all, let us note that (ii) and (c) follow from (i) and (b), applied to $-A$ and $-\product$ instead of $A$ and $\product$. Moreover, $0\in W_{\rm co}(A)$ if and only if there exists $x\in\dom A$ with $Ax\ne 0$ and $[Ax,Ax] = 0$. From this, (a) follows. Hence, we only need to prove (i), (iii), (b) and (d). Concerning (i), we note that
\begin{equation}\label{e:kerker}
\ker A = \ker A^2\quad\text{if $\ran A$ is negative.}
\end{equation}
To see this, let $x\in\ker A^2$. Then $[Ax,Ax] = [A^2x,x] = 0$ and thus $Ax = 0$ as $\ran A$ is negative.

The rest of the proof is divided into two steps. In the first step, we prove (i) and (iii) in the case when $A$ is bounded. Then, (i), (iii), (b), and (d) are proved successively in the unbounded case.

1. In this step we assume that $A$ is bounded. If $\ran A$ is negative, then $W_{\rm co}(A)\subset \mathbb R^-$ and, by \eqref{e:sigmas} and \eqref{e:Ilmenau}, we obtain $\sigma(A)\cap\mathbb R^+= \sigma(\wt A)\cap\mathbb R^+= \emptyset$. From \eqref{e:kerker} and $A\ne 0$ it follows that $\sigma(A)\cap\mathbb R^-\ne\emptyset$ (i.e. $\nu_-\in\R^-$ and $\mu_- > -\infty$), and as $\cal K$ is indefinite, $0\in \sigma(A)$ (cf.\ \eqref{HerzHerz}). Relation \eqref{17} now follows from Corollary \ref{cor}, and (i) is shown for bounded operators.

Assume that $\ran A$ is indefinite. If both $\sigma(A)\cap\R^\pm$ are non-empty, then (iii) follows from Corollary \ref{cor}. Let $\sigma(A)\cap\R^+ = \emptyset$. If also $\sigma(A)\cap\R^- = \emptyset$, then $A^2 = 0$ and hence $W_{\rm co}(A) = \{0\}$ as well as $\nu_- = \nu_+ = 0$, and (iii) holds. Assume that $\sigma(A)\cap\R^-\neq\emptyset$. If zero is not an isolated eigenvalue of $A$, then (iii) again follows from Corollary \ref{cor}. If zero is an isolated eigenvalue of $A$, then $\ker A\neq\ker A^2$, since otherwise $\ran A = E(\sigma(A)\cap\R^-)\calK$ is negative. Therefore, Corollary \ref{cor} yields $\ol{W_{\rm co}(A)} = [\nu_-,0] = [\nu_-,\nu_+]$, and (iii) is proved.

2. Let $A$ be unbounded. For $n\in\mathbb N$ let $E_n := E([-n,n])$ and consider the operator $A_n := A|E_n\calK$ in the Krein space $(E_n\calK,\product)$. Replace in \eqref{Grins1} $A$ by $A_n$ and denote the corresponding constants by $\mu_{\pm,n}$ and $\nu_{\pm,n}$, respectively. We deduce from the first step that
\begin{equation}\label{MIB3}
(\nu_{-,n},\mu_{-,n}) = W_{\rm co}(A_n)\setminus\{\nu_{-,n},\mu_{-,n}\}\subset W_{\rm co}(A)
\end{equation}
if $\ran A_n$ is negative and
\begin{equation}\label{ITGirl}
(\nu_{-,n},\nu_{+,n}) = W_{\rm co}(A_n)\setminus\{\nu_{-,n},\nu_{+,n}\}\subset W_{\rm co}(A)
\end{equation}
if $\ran A_n$ is indefinite.

(i). Assume that $\ran A$ is negative. Then from \eqref{HerzHerz} we conclude that $\sigma(A)\cap\mathbb R^+ = \emptyset$. Hence, $\nu_- = -\infty$ since $A$ is unbounded. Moreover, $\ran A_n\subset\ran A$ is negative, and we have $\mu_{-,n} = \mu_-$ for large $n$ as well as $\nu_{-,n}\to -\infty$ as $n\to\infty$. Therefore, \eqref{MIB3} implies
$$
(-\infty, \mu_-)\subset W_{\rm co}(A).
$$
As $\sigma(A)\cap\R^+ = \emptyset$, the point $\infty$ is not a critical point of $A$, cf.\ Remark \ref{r:infty}, and hence, for $x\in\calK$ we have $E_nx\to x$ as $n\to\infty$. As also $AE_nx = E_nAx\to Ax$ for $x\in\dom A$, we obtain
$$
W_{\rm co}(A) \subset (-\infty, \mu_-],
$$
and (i) is shown.

(iii). Assume now that $\ran A$ is indefinite and $\nu_- = -\infty$. By \eqref{HerzHerz}, $\sigma(A)\cap [0,\infty)$ is nonempty. If $\sigma(A)\cap [0,\infty)$ is bounded, then $\nu_{+,n} = \nu_+$ for large $n$ and $\nu_{-,n}\to -\infty$ as $n\to\infty$. Now, using \eqref{ITGirl} instead of \eqref{MIB3}, we can proceed similarly as above to conclude that (iii) holds.

If $\sigma(A)\cap [0,\infty)$ is unbounded, then we have $\nu_\pm =\pm \infty$ and with \eqref{ITGirl}
$$
W_{\rm co}(A) = \R.
$$
Hence (iii) is shown for the case $\nu_-=-\infty$. The proof for $\nu_+ = \infty$ is similar.

It remains to show (b) and (d). In order to prove the first part of (b) let $x$ be an eigenvector corresponding to $\mu_- \in\sigma_p(A)\setminus\{ 0\}$. Then $[Ax,Ax] = \mu_-[Ax,x]$ and $\mu_-\in W_{\rm co}(A)$ follows. Conversely, let
$\mu_- \in  W_{\rm co}(A)$. By (i) and (a), $\mu_-<0$. Remark \ref{r:infty} implies that $\infty$ is not a critical point of $A$. Therefore, the operator $B := A|E((-\infty,\mu_-])\calK$ is well-defined and $\mu_-\in W_{\rm co}(B)$. But $B$ is a boundedly invertible selfadjoint operator in the Hilbert space $(E((-\infty,\mu_-])\calK,-\product)$. Hence,
$$
\mu_- = \frac{[Bx,Bx]}{[Bx,x]}\quad\text{implies}\quad\mu_-^{-1} = \frac{[B^{-1}y,y]}{[y,y]},
$$
where $y = Bx$. This proves $\mu_-^{-1}\in\sigma_p(B^{-1})$ and hence $\mu_-\in\sigma_p(A)$.

It remains to prove that in both cases (i) and (iii) we have $\nu_-\in W_{\rm co}(A)$ if and only if $\nu_-\in\sigma_p(A)$.  For the rest of the proof we thus assume that $\ran A$ is not positive. First, assume that $\nu_- = 0$. 
Then it follows from (i),  \eqref{e:kerker} and $A\ne 0$ that 
$\ran A$ is indefinite. Hence, (a) implies $\nu_-\in W_{\rm co}(A)$. Suppose that $\nu_-\notin\sigma_p(A)$. Then Lemma \ref{l:good_point}(ii) implies that $\sigma(A) = \sp(A)$ which contradicts our assumption that $\calK$ be indefinite. Hence, $\nu_-\in\sigma_p(A)$ follows.

Let $\nu_-\in\R^-$. Clearly, if $\nu_-\in\sigma_p(A)$, then $\nu_-\in W_{\rm co}(A)$. Assume that $\nu_-\in W_{\rm co}(A)$. Then there exists $x\in\dom A$ such that $Ax\neq 0$ and $[Ax,Ax] = \nu_-[Ax,x]$. If $\nu_-$ is an isolated spectral point of $A$, then it is an eigenvalue, and nothing is to prove. Hence, there exists some $\la_0\in\sigma(A)\cap(\nu_-,0)$. Let $\la\in (\nu_-,\la_0)$ be arbitrary, and define
$$
E_1 := E([\nu_-,\la])\quad\text{and}\quad E_2 := E((\la,\infty))
$$
as well as $\calK_j := E_j\calK$, $A_j := A|\calK_j$, and $x_j := E_jx$, $j=1,2$. Then $A_1$ is a bounded selfadjoint operator in the Hilbert space $(\calK_1,-\product)$ which yields
$$
[Ax_1,Ax_1]\,\ge\,\nu_-[Ax_1,x_1].
$$
Moreover, $A_2$ is a non-negative operator in the Krein space $(\calK_2,\product)$ with $\sigma(A_2)\subset [\la,\infty)$ and whose range is not positive (since $\la_0\in\sigma(A_2)$). By (i) and (iii), we have $\ol{W_{\rm co}(A_2)}\subset [\la,\infty)$. Hence,
$$
[Ax_2,Ax_2]\,\ge\,\la[Ax_2,x_2].
$$
Therefore, we obtain
\begin{align*}
\la[Ax_2,x_2] + \nu_-[Ax_1,x_1]
&\le[Ax_2,Ax_2] + [Ax_1,Ax_1]\\
&= [Ax,Ax] = \nu_-[Ax,x]\\
&= \nu_-[Ax_2,x_2] + \nu_-[Ax_1,x_1].
\end{align*}
As this implies $Ax_2 = 0$ (cf.\ \eqref{e:carsten}), we have $x\in E([\nu_-,\la])\calK\,[\ds]\,\ker A$ for every $\la\in (\nu_-,\la_0)$. Letting $\la\downto\nu_-$ gives
$$
x\in\ker(A - \nu_-)\,[\ds]\,\ker A.
$$
Since $Ax\neq 0$, this proves that $\ker(A - \nu_-)\neq\{0\}$ and thus $\nu_-\in\sigma_p(A)$.
\end{proof}

We close the paper with the following spectral inclusion theorem which follows 
from our two main results on the numerical range and the co-numerical range of 
a non-negative operator in a Krein space.

\begin{thm}\label{t:spec_inc}
Let $A\neq 0$ be a non-negative operator in the indefinite Krein space $(\calK,\product)$.
 If zero is an isolated eigenvalue of $A$ such that $\ker A = \ker A^2$ and that either
 $\mathbb R^+$ or $\mathbb R^-$ contains no spectrum of $A$, then
\begin{equation}\label{e:fast}
\sigma(A)\setminus\{0\}\,\subset\,\ol{W(A)\cap W_{\rm co}(A)}.
\end{equation}
In all other cases we have
\begin{equation}\label{e:ganz}
\sigma(A)\,\subset\,\ol{W(A)\cap W_{\rm co}(A)}.
\end{equation}
\end{thm}
\begin{proof}
In the sequel we will frequently make use of the following implication which directly 
follows from Theorems \ref{t:w} and \ref{t:cow}:
\begin{equation}\label{e:ntu}
\ran A\,\text{ not negative}\qquad\Lra\qquad (\mu_+,\nu_+)\subset W(A)\cap W_{\rm co}(A).
\end{equation}
Let $\la\in\sigma(A)\setminus\{0\}$. If $\la\notin\{\mu_+,\nu_+,\mu_-,\nu_-\}$, then 
 by the Theorems \ref{t:w} and \ref{t:cow}
 $\la\in W(A)\cap W_{\rm co}(A)$. Let $\la\in\{\mu_+,\nu_+\}$. 
Then $\sigma(A)\cap\R^+\neq\emptyset$ and hence $\ran A$ is not negative. Therefore, 
\eqref{e:ntu} implies $(\mu_+,\nu_+)\subset W(A)\cap W_{\rm co}(A)$ which shows 
$\la\in\ol{W(A)\cap W_{\rm co}(A)}$ unless $\mu_+ = \nu_+$. But then $\la$ is an 
eigenvalue of $A$, and the same holds. A similar argument applies to the case 
$\la\in\{\mu_-,\nu_-\}$, and \eqref{e:fast} is proved.

Let $0\in\sigma(A)$. For \eqref{e:ganz} it remains to prove that $0\notin\ol{W(A)\cap W_{\rm co}(A)}$ 
implies that zero is an isolated eigenvalue of $A$, $\ker A = \ker A^2$ and either 
$\sigma(A)\cap\R^-$ or $\sigma(A)\cap\R^+$ is empty.
Suppose that zero is not an isolated point of $\sigma(A)$. Then $\mu_+ = 0$ or 
$\mu_- = 0$. Assume, e.g., $\mu_+ = 0$. Then $\ran A$ is not negative and hence 
$(0,\nu_+)\subset W(A)\cap W_{\rm co}(A)$ by \eqref{e:ntu}, which is a contradiction. 
Consequently, zero is an isolated eigenvalue of $A$. If $\ran A$ is indefinite, 
then $0\in W(A)\cap W_{\rm co}(A)$ by Theorem \ref{t:w}(a) and Theorem \ref{t:cow}(a). 
Therefore, $\ran A$ is definite and the rest follows easily from  \eqref{e:kerker}
and   \eqref{HerzHerz}.
\end{proof}

\section{Conclusions}
We studied and characterized the (Krein space) numerical range of a possibly unbounded non-negative operator $A$ in a Krein space. We proved that the numerical range is never bounded from below or from above. If the Krein space inner product is indefinite on $\ker A$, but not neutral, then the numerical range of $A$ even coincides with the entire real axis and, in particular, does not provide any information on the location of the spectrum. For this reason we introduced the co-numerical range of $A$ which is another subset of the real numbers associated with the operator $A$. In contrast to the numerical range, the co-numerical range of $A$ is always bounded from above (from below) if the spectrum of $A$ is bounded from above (from below, respectively). Moreover -- with the exception of a very special case -- its closure also contains the spectrum of $A$, and we have the spectral inclusion $\sigma(A)\subset\ol{W(A)\cap W_{\rm co}(A)}$.

\section*{Contact information}
Friedrich Philipp: Institut f\"ur Mathematik, Technische Universit\"at Berlin, Stra\ss e des 17.\ Juni 136, 10623 Berlin, Germany, philipp@math.tu-berlin.de

\vspace{0.4cm}\noindent
Carsten Trunk: Institut f\"ur Mathematik, Technische Universit\"at Ilmenau, Postfach 10 05 65, 98684 Ilmenau, Germany, carsten.trunk@tu-ilmenau.de
\end{document}